\documentclass[12pt]{amsart}
\usepackage{amssymb,amsmath}
\usepackage{amssymb,amsfonts}
\usepackage{amsmath,latexsym}

\newtheorem{proposition}{Proposition}[section]
\newtheorem{lemma}{Lemma}[section]
\newtheorem{definition}{Definition}[section]

\def\refer#1{~\ref{#1}}
\def\refeq#1{~(\ref{#1})}
\def\ccite#1{~\cite{#1}}

\def\longformule#1#2{
\displaylines{
\qquad{#1}
\hfill\cr
\hfill {#2}
\qquad\cr}}

\let\e=\varepsilon

\let\D=\Delta
\let\Lam=\Lambda

\def\virgp{\raise 2pt\hbox{,}}
\def\cdotpv{\raise 2pt\hbox{;}}

\def\eqdefa{\buildrel\hbox{\footnotesize def}\over =}
\def\Id{\mathop{\rm Id}\nolimits}

\def\C{\mathop{\mathbb C\kern 0pt}\nolimits}
\def\DD{\mathop{\mathbb D\kern 0pt}\nolimits}
\def\EE{\mathop{\mathbb E\kern 0pt}\nolimits}
\def\K{\mathop{\mathbb K\kern 0pt}\nolimits}
\def\N{\mathop{\mathbb  N\kern 0pt}\nolimits}
\def\Q{\mathop{\mathbb  Q\kern 0pt}\nolimits}
\def\R{\mathop{\mathbb R\kern 0pt}\nolimits}
\def\SS{\mathop{\mathbb  S\kern 0pt}\nolimits}
\def\Z{\mathop{\mathbb  Z\kern 0pt}\nolimits}
\def\ZZ{\mathop{\mathbb  Z\kern 0pt}\nolimits}
\def\H{\mathop{\mathbb  H\kern 0pt}\nolimits}
\def\PP{\mathop{\mathbb P\kern 0pt}\nolimits}
\def\TT{\mathop{\mathbb T\kern 0pt}\nolimits}

\def\dive{\mathop{\rm div}\nolimits}

\def\tN#1{|\kern -0.05em\|{#1}|\kern -0.05em\|}

\def\hat{\widehat}

\numberwithin{equation}{section}
\newtheorem{thm}{Theorem}

\newtheorem{remark}[lemma]{Remark}

\newtheorem{prop}[lemma]{Proposition}

  \newcommand{\cqfd}{\mbox{ } \hfill$\Box$}

  \def\e{\varepsilon} \def\cdotv{\raise 2pt\hbox{,}}
   \def\eqdefa{\buildrel\hbox{\footnotesize
      def}\over =}

  \let\ds=\displaystyle

  \newcommand{\LB}[4]{\mathcal{L}^{#1}_t(\dot{B}^{{#2},{#4}}_{#3})}
  \newcommand{\Be}[3]{\dot{B}^{{#1},{#3}}_{#2}}
  \newcommand{\Lb}[2]{\mathcal{L}^{#1} {B}_{#2}}
  \newcommand{\Lh}[2]{\mathcal{L}^{#1} {B}^{#2}}

   \def\pasdegrille{\let\grille =
    \pasgrille}  \def\aat#1#2#3{
    \divide \dimen1 by 48 \dimen3=\dimen1 \multiply \dimen1 by #1
    \advance \dimen1 by -\dimen3 \divide \dimen1 by 101 \multiply
    \dimen1 by 100 \divide \dimen2 by \count11 \multiply \dimen2 by #2
    \setbox0=\hbox{#3}\ht0=0pt\dp0=0pt \rlap{\kern\dimen1 \vbox
      to0pt{\kern-\dimen2\box0\vss}}\dimen1= \wd1 \dimen2=\ht1}
  \def\pasgrille{ \count12= \dimen1 \divide \count12 by 50 \divide
    \dimen2 by \count12 \count11 =\dimen2 \ \divide \dimen1 by 48
    \setlength{\unitlength}{\dimen1} \smash{\rlap{\ }} \dimen1= \wd1
    \dimen2=\ht1 } \def\grille{ \count12= \dimen1 \divide \count12 by
    50 \divide \dimen2 by \count12 \count11 =\dimen2 \ \divide \dimen1
    by 48 \setlength{\unitlength}{\dimen1}
    \smash{\rlap{\graphpaper[1](0,0)(50, \count11)}} \dimen1= \wd1
    \dimen2=\ht1 }

  \edef\@tempa#1#2{\def#1{\mathaccent\string"\noexpand\accentclass@#2
    }} \@tempa\ring{017}
  
\begin{document}

  \title{Self-improving bounds for the Navier-Stokes equations}

  \author{Jean-Yves Chemin}
  \address{ Laboratoire Jacques-Louis Lions, UMR CNRS 7598\\
    Universit\'e Pierre et Marie Curie\\
    4 place Jussieu 75005 Paris\\
    France} \email{chemin@ann.jussieu.fr}

  \author{Fabrice Planchon}
  \address{Laboratoire J. A. Dieudonn\'e, UMR CNRS 7351\\
    Universit\'e de Nice Sophia-Antipolis\\
    Parc Valrose\\
    06108 Nice Cedex 02\\
    France and Institut universitaire de France} \email{fabrice.planchon@unice.fr} \thanks{The second
    author was partially supported by A.N.R. grant SWAP} \date{}
  \begin{abstract}
    We consider regular solutions to the Navier-Stokes equation and
    provide an extension to the Escauriaza-Seregin-Sverak blow-up
    criterion in the negative regularity Besov scale, with regularity
    arbitrarly close to $-1$. Our results rely on turning a priori
    bounds for the solution in negative Besov spaces into bounds in
    the positive regularity scale.

  \end{abstract}
  \maketitle
  \section{Introduction}
  We consider the incompressible Navier-Stokes equations in
  $\mathbb{R}^3$,
  \begin{equation*}
\tag{NS} \quad \left \{ \begin{array}{rcl} \partial_t u & = & \Delta u - \nabla \cdot (u\otimes u)-\nabla \pi,\\ 
\dive u & = & 0,\\ u|_{t=0} & = & u_{0}
      \end{array} \right.
  \end{equation*}
  for $(x,t)\in \mathbb{R}^3 \times \R^+$, where
  $u=(u_{i}(x,t))_{i=1}^3 \in \mathbb{R}^3$ is the velocity vector
  field, $\pi(x,t)\in \R$ is the associated pressure function and
  \begin{equation*}
 {\nabla \cdot (u \otimes u)} :=
  \Bigl(\sum_{j=1}^d \partial_{x_{j}}(u_{i}u_{j})\Bigr)_{i=1}^d.
  \end{equation*}
In the pioneering work\ccite{leray}, J. Leray proved the existence of
global turbulent (weak in the modern terminology) solutions of~(NS)
for initial data with finite kinetic energy, i.e. initial data
in~$L^2$. These solutions need not be unique or preserve regularity of the initial data. In this same
  work, J. Leray proved  that for regular enough initial data (namely~$H^1$ initial data), a local (in time) unique solution
  exists. He also proved that as long as this solution is regular
  enough, it is unique among all the possible turbulent solutions, and
  moreover, if such a turbulent solution satisfies
  \begin{equation}
    \label{eq:serrin}
    u\in L^p([0,T[;L^q(\mathbb{R}^3)) \text{ with } \frac 2 p +\frac 3 q=1,\,\,q>3,
  \end{equation}
  then the solution remains regular on~$[0,T]$ and can be extended
  beyond time~$T$. This is now known as Serrin's criterion.
  
  On the other hand, there is a long line of works on constructing
 local in time solutions, from H. Fujita and~T. Kato (see\ccite{KF})
 to H. Koch and D. Tataru (see\ccite{KT}). For these results, the main
 feature is that the initial data belongs to spaces which are
 invariant under the scaling of the equations. Between\ccite{KF}
 and\ccite{KT}, T. Kato (see\ccite{kato}) proved wellposedness of (NS) for initial data~$u_0$  in~$L^3$.  In this
  framework of local in time (strong, e.g. unique) solutions, Serrin's
  criterion may be understood as a non blow-up criterion at time $T$:
  e.g. if $u$ is a strong solution with $u_{0}\in L^3({\mathbb R}^3)$,
  that is $u\in C([0,T[;L^3({\mathbb R^3}))$, and if\refeq{eq:serrin}
  is satisfied, then one may (continuously and uniquely) extend the
  solution $u$ past time $T$.

  In the recent important work \cite{ess}, L. Escauriaza, G. Seregin and V. \v
  Sver\'ak obtained the endpoint version of Serrin's criterion, using
  blow-up techniques to construct a special solution vanishing at
  blow-up time and then backward uniqueness to rule out its
  existence. Earlier work of Giga and Von Wahl proved this endpoint
  under a continuity in time assumption in $L^3$, and such a
  continuity result was recently improved to match the local in time
  theory by Cheskidov-Shvydkoy \cite{Ches}.

  Our first theorem (Theorem \ref{th1bis} below) may be seen as an
  extension of the endpoint criterion by Escauriaza-Seregin-\v
  Sver\'ak, in the negative regularity scale. Before providing an
  exact statement, we need to introduce a few notations and
  definitions.

  Since we are interested in smooth (or at least strong in the Kato
  sense) solutions, (NS) is equivalent for our purpose with its
  integral formulation, where the pressure has been disposed of with
  the projection operator $\mathbb{P}$ over divergence free vector
  fields:
  \begin{equation}
    \label{eq:nsi}
    u=S(t)u_{0}-\int_{0}^t \mathbb{P} S(t-s)\nabla\cdot (u\otimes u)(s)\,ds=u_{L}+B(u,u)  
  \end{equation}
  where $S(t)=\exp(t\mathbb{P}\Delta)=\mathbb{P}\exp(t\Delta)$ is the Stokes
  flow (which is nothing but the heat flow in~$\mathbb{R}^3$ on
  divergence free vector fields) and $B(u,u)$ is the Duhamel term
  which reads, component wise
  \begin{equation}
    \label{eq:defB}
    B(f,g)=-\int_{0}^t R_{j} R_{k} R_{l} |\nabla| S(t-s) (fg)(s)\,ds,
  \end{equation}
  where the $R_{(\cdot)}$ are the usual Riesz transforms (recall
  $\mathbb{P}$ is a Fourier multiplier with matrix valued symbol
  $\Id-|\xi|^{-2} \xi \otimes \xi$). We will denote the Lebesgue norm
  by
$$
\|f\|_{p}=\|f\|_{L^p}=\left(\int_{\R^3} |f(x)|^p\,dx\right)^{\frac 1
  p}.
$$
Let us recall a definition of Besov spaces using the heat flow
$S(\sigma)$.
\begin{definition}
  \label{d1}
  Let $Q(\sigma)=\sigma\partial_\sigma S(\sigma)$. We define $\dot
  B^{s,q}_{p}$ as the set of tempered distributions~$f$ such that
  \begin{itemize}
  \item the integral $ \int^{N}_{1/N} Q(\sigma) f \, d\sigma/\sigma $
    converges to $f$ when $N\rightarrow +\infty$ as a tempered distribution if~$s < {\frac{d}{p}}$
    and after taking the quotient with polynomials if not, and
  \item the function $\sigma^{-s/2}\| Q(\sigma) f \|_{p}$ is in
    $L^q(d\sigma/\sigma)$; its norm defines the Besov norm of $f$:
    \begin{equation}
      \label{eq:defBesovHeat}
      \|f\|^q_{\dot B^{s,q}_p}=\int_0^{+\infty} \sigma^{-sq/2}\|Q(\sigma) f\|^q_p \,\frac{d\sigma}\sigma.
    \end{equation}
  \end{itemize}
\end{definition}
We recall that the usual (homogeneous) Sobolev spaces $\dot H^s$,
defined through the Fourier transform by $|\xi|^s \hat f(\xi) \in
L^2$, may be identified with $\dot B^{s,2}_2$, while the critical
Sobolev embedding holds as follows: $\dot B^{s,q}_p\hookrightarrow
\dot B^{\rho,\lambda}_r$ provided $s-d/p=\rho-d/r$, $s\geq \rho$ and
$q\leq \lambda$, as well as~$\dot B^{s,q}_p \hookrightarrow L^{r}_x$
provided $s-d/p=-d/r$, $s\geq 0$ and $q\leq r$.

We are now in a position to state our first result:
\begin{thm}
  \label{th1bis}
  Let $u$ be a local in time solution to (NS) such that
  $u_{0}\in \dot H^{1/2}$. Assume that there exist $p\in ]3,+\infty[$
  and $q<2p'$ such that
  \begin{equation}
    \label{eq:apriori2}
    \sup_{t\in [0,T[} \| u(\cdot,t)\|_{ \Be { 3/p-1} p {q}} <+\infty,
  \end{equation}
  then the solution may be uniquely extended past time $T$.
\end{thm}

We remark that our hypothesis allows for smooth, compactly supported
data; actually, one may simply assume that the vorticity~$\omega_0=\nabla \wedge u_0$ belongs to~$L^{3/2}$. By Sobolev embedding and the Biot-Savart law, this
implies that~$u_0$ belongs to~$\dot H^\frac 1 2 \subset \Be { 3/p-1} p {2}$. Hence
by local Cauchy theory so does $u$ and \eqref{eq:apriori2} is finite
at least for small times.

It is of independent interest to consider the case of $L^3$ data,
without any extra regularity hypothesis:
\begin{thm}
  \label{th2}
  Let $u$ be a local in time strong solution to (NS) with data
  $u_{0}$ in~$ L^3 \cap \Be { 3/p-1} p q$, with $3<p<+\infty$ and $q<
  2p'$. Assume that
  \begin{equation}
    \label{eq:apriori2bis}
    \sup_{t\in [0,T[} \| u(\cdot,t)\|_{ \Be
      { 3/p-1} p q} <+\infty,
  \end{equation}
  then the solution may be uniquely extended past time $T$.
\end{thm}
The restriction on $q$ for the data implies that $q<3$ as $p>3$. As
such, our result does not include the $L^3$ case, as we are still
assuming a subtle decay hypothesis through the $q$ indice. However,
the restriction is mostly technical and all is required to lift it is
to generalize the results from \cite{GKP}, most specifically the
compactness result which is only stated in $L^3$ rather than in
the Besov scale. This will be adressed elsewhere, providing
generalizations of the present note and the results of \cite{GKP}. Our
purpose here is to illustrate that these blow up criterions do not
require positive regularity on the data; in fact, they will extend to
non $L^3$ data into the negative Besov scale.

Both Theorem \ref{th1bis} and \ref{th2} rely crucially on improving
the rather weak a priori bound on~$u$ from the hypothesis. Such
``self-improvements'' are of independent interest and we state
examples of them below. We start with a (spatial) regularity
improvement for negative Besov-valued data (see the forthcoming Remark
\ref{rem-it} on the $p$ range restriction which is only technical).
\begin{thm}
  \label{th1}
  Let $u$ be a local in time strong solution to (NS) with data
  $u_{0}\in \Be { 3/p-1} p q$, with~$3<p<6$ and $q<+\infty$. Assume
  that
  \begin{equation}
    \label{eq:apriori}
    \sup_{t\in [0,T[} \| u(\cdot,t)\|_{ \Be
      { 3/p-1} p \infty} \leq M,
  \end{equation}
  then we have the following improved uniform bound on
  $w_3=u-u_{L}-B(u_L,u_L)$,
 \begin{equation}
 \label{eq:iap}
 \sup_{t\in [0,T[} \| w_3(\cdot,t)\|_{ \Be{ \frac 1 2} 2 \infty} \leq C(M),
 \end{equation}
  where $C$ is an explicit smooth function of its argument.
\end{thm}
For any initial datum $u_{0}\in \dot{B}^{{-(1-3/p)},q }_{p}$, with
$1\leq p, q<+\infty$, there exists a unique, local in time, strong
solution to (NS). Such solutions were obtained in \cite{Cannone}
for $3<p\leq 6$ and for all finite~$p$ in\ccite{Planchon}, and we
refer to the appendix of \cite{gip3} for a proof which is closer in
spirit to the present note. One should point out that all these Besov
spaces are embedded in $VMO^{-1}$ (limits of smooth, compactly
supported functions in $BMO^{-1}$) and that strong solutions in this
endpoint space were obtained in \cite{KT}.

Strong solutions are known to obey the same space-time estimates as
the heat flow on any compact time interval on which they exist: one
may take advantage of these estimates to improve regularity on
$w=u-u_{L}$ in this context, as was done in \cite{CP} for $L^3$
data and in \ccite{Planchon,gip3} for $\Be {{3/p-1}} p q$ by
substracting further iterates of the heat flow. However, to our
knowledge, the only known result assuming an a priori bound with no
time integrability was proved in \ccite{CP} where the conclusion of
Theorem \ref{th1} is obtained assuming a slightly weaker condition than~$u\in
L^\infty_{t} L^{3}_{x}$ (the Lebesgue space is replaced by its larger
weak counterpart).

Finally, we provide a time regularity improvement, whose proof can be used
to obtain Theorem \refer{th1} in the range~$p\leq 4$, but should be of
independent interest.
\begin{thm}
  \label{thholder}
  Let $u$ be a local in time strong solution to (NS) with data
  $u_{0}\in \Be { -1/4} 4 4$. Assume that
  \begin{equation}
  \label{eq:apriorider}
   \sup_{t\in [0,T[} \| u(\cdot,t)\|_{ \Be { -1/4} 4 4} \leq M,
  \end{equation}
  then~$u$ has the following H\"older in time regularity: 
    \begin{equation}
 \label{eq:iapbis}
 \forall  (t,t')\in [0,T[^2\,,\ \|u(\cdot,t)-u(\cdot, t')\|_{\Be
   {-3/4}  4 4}
\leq C(M)|t-t'|^{\frac 1 4}.
  \end{equation}
\end{thm}

For notational convenience, set, for any $1\leq p$,
$$
B_{p} = \dot B^{-(1-3/p),\infty}_{p}, \,\,\, B^s=B_{p} \,\,\text{ with
}s-3/p=-1.
$$
In other words, indices are tight by scaling and we indifferently use
regularity or decay to label spaces with scale $-1$. In what follows,
we shall also need a suitable modification of Besov spaces, taking
into account the time variable.
\begin{definition}
  \label{raah11}
  For $1\leq \rho \leq +\infty$, we shall say that $u(x,t)$ belongs
  to~$ {\mathcal{L}^{\rho}}([ a,b ] ;\dot{B}^{s,q}_{p})$ if~$u(t)$ is in~$ \dot B^{s,q}_{p}$ for all $t\in [a,b]$ and
  \begin{equation*}
    % \label{eq:raaah12}
    \int_0^{+\infty} \sigma^{-sq}\|Q(\sigma) u\|^q_{L^\rho([ a,b
      ];L^{p}_{x})} \, \frac{d\sigma}{\sigma} <+\infty\ .
  \end{equation*}
  The associated norm is defined in the obvious way and $
  {\mathcal{L}^{\rho}_{T}}(\dot{B}^{s,q}_{p}) :=
  {\mathcal{L}^{\rho}}([ 0,T ] ;\dot{B}^{s,q}_{p})$.
\end{definition}
As before, we will adopt the following shorthand notation
$$
\Lb \rho p = \Lh \rho s = \LB \rho s p \infty \,\,\text{ with }\,
s=-1+3/p+2/\rho,
$$
which is consistent with the previous one: $L^\infty_{t} B_{p}=\Lb
\infty p$.

We will denote by $\lesssim$ a less or equal sign with a harmless
constant, and $C$ any irrelevant constant which may change from line
to line.

\section{From a priori bounds to a generalized endpoint Serrin's criterion}
From Sobolev's embedding, Theorem \ref{th1bis} immediately follows from
Theorem \ref{th2}. In turn, Theorem \ref{th2} is a consequence of the
following key proposition.
\begin{prop}
  \label{keyprop}
Let $u$ be as in Theorem \ref{th2}. Then there exists a decomposition
$u=v+w$ such that
\begin{equation}
  \label{eq:splitv}
  \sup_{t\in [0,T[}\|v(\cdot,t)\|_{L^3\cap \dot B^{3/p-1,q}_p}\leq C(M,u_0),
\end{equation}
\begin{equation}
  \label{eq:splitw}
\sup_{t\in [0,T[}\|w(\cdot,t)\|_{ B_{1/(1-\varepsilon)}} \leq C(M,u_0),
\end{equation}
where $\varepsilon$ may be chosen arbitrarly small.
\end{prop}
Postponing the proof of this proposition for a moment, we prove Theorem
\ref{th2}: notice that \refeq{eq:splitv} provides an a priori bound
for the $v$ part in $L^\infty([0,T[;L^3)$; we seek to obtain a similar
bound for the $w$ part. As $w=u-v$, we also have a bound on $w$ in
$L^\infty([0,T[;\dot B^{3/p-1,q}_p)$, from \eqref{eq:splitv} and
\eqref{eq:apriori2bis}. As~$q<2p'$, let us write
$$
q= \frac {2p'} {1+\eta} \quad {\rm with} \  \eta\ {\rm small\ enough.}
$$
Then define
$$
r=\frac 3 {1+2\eta}\,\virgp\ \ \theta=\frac{1+2\eta-3/p}{3(1/p'-3\varepsilon)}\quad\hbox{and}\quad b = \frac q {1-\theta}\,\virgp
$$
and notice that~$b\leq 3$.
We now combine this bound with \eqref{eq:splitw}, using
convexity of norms and Sobolev embedding of Besov spaces into
Lebesgue ones. This gives
$$
\|w\|_{L^3}\lesssim \|w\|_{\dot B^{(3/r-1),b}_{r}} \lesssim \|w\|^\theta_{ B_{1/(1-\varepsilon)}} \|w\|^{1-\theta}_{\dot B^{(3/p-1),q}_{p}}.
$$
As such, we have obtained control of $u=v+w$ in $L^\infty_t(L^3_x)$,
which allows to use the Escauriaza-Seregin-\v Sver\'ak result to conclude
the proof of Theorem \ref{th2}.\cqfd

\medbreak
We now prove Proposition \ref{keyprop}. Note that a local in time solution
with data in $\dot B^{3/p-1,q}_p$ exists and additional regularity is preserved (see for instance\ccite {chemin20}
or\ccite{gip3}). Hence
we do not worry about existence, but rather focus on improving
bounds. It is convenient to present the argument in a rather abstract
 setting. Recall $B$ was defined in \eqref{eq:nsi}, and set
 $w_{2}=u-u_L=B(u,u)$, then
 \begin{equation}
   \label{eq:w2}
w_{2}=B(u_{L},u_{L})+2 B(u_{L},w_{2})+B(w_{2},w_{2})   
 \end{equation}
where we are obviously abusing notations (writing
$B(u,v)=B(v,u)$). Note that from a priori bound
\eqref{eq:apriori2} and local existence theory, we have $u_{L}\in
L^\infty_{t} \dot B^{-(1-3/p),q}_{p}$ with a uniform bound $2M$, while
obviously $u_L \in C_t(L^3_x)$ with bound $\|u_0\|_{L^3}$.

We start with an easy case which
already provides the key features of
the general argument without technicalities.
\begin{lemma}\label{cas1}
  Assume in addition to the hypothesis of Theorem \ref{th2} that
  $\omega_0\in L^{3/2}$; then Proposition \ref{keyprop} holds with
  $v=u_L$, $w=w_2$ and $\varepsilon=0$.
\end{lemma}
We just remarked that, even without additional requirements,
\eqref{eq:splitv} holds for $v=u_L$. We are left with proving
\eqref{eq:splitw} for $w_2$: we will use \eqref{eq:w2}. Note that by
the Biot-Savart law,~$\nabla u_0$ belongs to~$L^{3/2}$ and thus~$\nabla u_L$ to~$C_t(L^{3/2})$. By chain rule,
$\nabla_x(u_L\otimes u_L)$ is in~$L^\infty_t(L^1)$. Using  Proposition
4.1 in\ccite{gip3}, we infer
\begin{equation}
  \label{eq:presquefiniter}
\|B(u_{L},u_{L})\|_{\Lb \infty 1} \lesssim \|u_0\|_{L^3} \|\nabla u_0\|_{L^{\frac 32}}.
\end{equation}
Therefore, we seek an a priori bound for $w_2$ in $\Lb \infty 1$ from
the weaker bound \eqref{eq:apriori2} on $u$.

To deal with the remaining terms in \eqref{eq:w2}, we use
the following lemma:
\begin{lemma}\label{productlemma}
 Let $1\leq r< 3<p<+\infty$, $f,g\in \Lb \infty {r}\cap \Lb \infty
 p$,  $2/3<1/r+1/p\leq 1$ and $1/\eta\leq 1/r+1/p$, then
  \begin{equation}
    \label{eq:bootfaciledual}
    \| B(f,g)\|_{\Lb \infty \eta } \lesssim \| f\|_{\Lb \infty {r}} \|g
    \|_{\Lb \infty {p}}+ \| g\|_{\Lb \infty {r}} \|f\|_{\Lb \infty {p}}\,.
  \end{equation}
If $p=3$, the same estimate holds with $B_3$ replaced by $L^3_x$,
  \begin{equation}
    \label{eq:bootfacile}
    \| B(f,g)\|_{\Lb \infty \eta} \lesssim \| f\|_{\Lb \infty {r}}\|g\|_{L^\infty_t(L^3_x)}+ \| g\|_{\Lb \infty {r}}\|f\|_{L^\infty_t(L^3_x)}\,.
  \end{equation}
\end{lemma}
The proof of the lemma follows directly from standard product rules in
Besov spaces and properties of the operator $B$ defined by
\eqref{eq:defB}, see e.g. Proposition 4.1 in \cite{gip3}.\cqfd

\medbreak
For the term $B(w_2,w_2)$, \eqref{eq:bootfaciledual} yields
\begin{equation}
  \label{eq:presquefiniavant}
  \|B(w_2,w_2)\|_{\Lb \infty 1} \lesssim \|w_2\|_{\Lb \infty p} \| w_2 \|_{\Lb
    \infty {p'}}\,,
\end{equation}
and by convexity of Besov norms,
$$
\|w_2\|_{\Lb \infty {p'}} \lesssim \|w_2\|^\lambda_{\Lb \infty
  {1}}\|w_2\|^{(1-\lambda)}_{\Lb \infty {p}}, \,\,\text{ with }\,\, \lambda=\frac{p-2}{p-1}\,\cdotp
$$
Therefore,
\begin{equation}
  \label{eq:presquefini}
  \|B(w_2,w_2)\|_{\Lb \infty 1} \lesssim K^{2-\lambda} \| w_2 \|_{\Lb
    \infty {1}}^\lambda\quad\hbox{with}\quad K = \sup_{t\in [0,T[} \|u(\cdot,t)\|_{B_p}\,.
\end{equation}
The crossterm is handled in a similar way: convexity of norms yields again
$$
\|w_2\|_{\Lb \infty {3/2}} \lesssim \|w_2\|^\eta_{\Lb \infty
  {1}}\|w_2\|^{(1-\eta)}_{\Lb \infty {p}}, \,\,\text{ with }\,\, \eta=\frac{2p-3}{3(p-1)}\,\virgp
$$
and by \eqref{eq:bootfacile}
\begin{equation}
\label{eq:presquefinibis}
  \|B(u_L,w_2)\|_{\Lb \infty 1} \lesssim \|u_0\|_{L^3}
K^{1-\eta} \| w_2\|_{\Lb \infty 1}^\eta + \|\nabla u_0\|_{L^{\frac 32}}
K^{1-\theta} \|w_2\|^\theta_{\Lb \infty 1}.
\end{equation}

Gathering \eqref{eq:presquefiniter}, \eqref{eq:presquefinibis} and~\eqref{eq:presquefini} and
using convexity, we obtain the desired control of $w_2$ in~$\Lb \infty
1$, which ends the proof of Lemma \ref{cas1}.\cqfd

\medbreak
In order to lower the regularity requirement on $u_0$, we need to deal with the crossterm in a different
way: in fact, the part of $B(u_L,w_2)$ which carries high frequencies
of $u_L$ has no reason to be any better than
$B_{3/(2(1-\varepsilon))}$. Hence, we seek first such an a priori
estimate for~$w_2$, and then bootstrap
this intermediate estimate to a suitable estimate in $B_{1/(1-\varepsilon)}$ for the
next term in the expansion:
\begin{lemma}
  \label{cas2}
 Under  the hypothesis of Theorem\refer{th2}, Proposition \ref{keyprop} holds with 
 $$
 v=u_L+B(u_L,u_L)+2B(u_L,w_2)\quad\hbox{and}\quad w=B(w_2,w_2).
 $$
\end{lemma}
Remark that, by standard heat estimates, the bound \eqref{eq:splitv}
holds for $B(u_L,u_L)$ as it already does for $u_L$. We now use the following lemma to take care of the crossterm:
\begin{lemma}\label{crosstermL3B}
 Let $3<p<+\infty$,  $f\in \Lb \infty p$, then
  \begin{equation}
    \label{eq:crosstermL3}
   \| B(u_L,f)\|_{L^\infty_t(L^3)} \lesssim \| f\|_{\Lb \infty p} \|u_0
   \|_{L^3}\,,
  \end{equation}
and
  \begin{equation}
    \label{eq:crosstermBp}
   \| B(u_L,f)\|_{L^\infty_t(\dot B^{-(1-3/p),q}_p)} \lesssim \| f\|_{\Lb \infty p} \|u_0
   \|_{\dot B^{-(1-3/p),q}_p}\,.
  \end{equation}
\end{lemma}
The proof of Lemma \ref{crosstermL3B} follows once again from product
rules and properties of $B$ (Proposition 4.1, \cite{gip3}), provided one uses heat estimates on
$u_L$: for \eqref{eq:crosstermL3}, one uses $u_L \in \mathcal{L}^{2p'}_t
(\dot B^{1/p',3}_3) $, while for \eqref{eq:crosstermBp} one uses  $u_L
\in \mathcal{L}^{1}_t (\dot B^{3/p+1,q}_p)\cap \mathcal{L}^{\infty}_t (\dot B^{3/p-1,q}_p) $.\cqfd

\medbreak
 We now apply the lemma to
$f=w_2$ (which was already proved to be in $\Lb \infty p$) and finally
get bound \eqref{eq:splitv} on our $v=u_L+B(u_L,u_L)+2B(u_L,w_2)$. We
now turn to the bound on $w$, with a new product lemma:
\begin{lemma}\label{productweakL}
 Let $\varepsilon>0$ be small, $f\in \Lb \infty {3/(1+\varepsilon)}$, then
  \begin{equation}
    \label{eq:crosstermB32}
    \| B(u_L,f)\|_{\Lb \infty {3/(2(1-\varepsilon))} } \lesssim \| f\|_{\Lb \infty {3/(1+\varepsilon)}} \|u_0
   \|_{L^3}\,.
  \end{equation}
\end{lemma}
As before, the lemma follows from product rules in Besov spaces,
actually requiring only~$u_L \in \Lh \infty {-3\e}$.\cqfd

\medbreak
Going back to \eqref{eq:w2}, we have $B(u_L,u_L)\in \Lh \infty 1$
from standard estimates (or suitable tweaking of the previous lemma,
or \cite{CP}). From Lemmata \ref{productlemma} and \ref{productweakL},
\begin{equation*}
  \|w_2\|_{\Lb \infty {3/(2(1-\varepsilon))} } \lesssim C(u_0)+ \| w_2\|_{\Lb \infty {3/(1+\varepsilon)}} \|u_0
   \|_{L^3}+\|w_2\|^2_{\Lb \infty {3/(1+\varepsilon)} }
\end{equation*}
and by convexity of Besov norms, 
\begin{equation*}
  \|w_2\|_{B_{3/(1+\varepsilon)}} \leq \|w_2\|^\lambda_{B_{3/(2(1-\varepsilon))}}\|w_2\|^{1-\lambda}_{B_{p}}\,,
\end{equation*}
where $\lambda=((1+\epsilon)p-3)/(2(1-\epsilon)p-3)
<1/2$, provided $\varepsilon< 3/(4p)$. Hence, combining the three previous inequalities and convexity,
we obtain
$$
\|w_2\|_{\Lb \infty {3/(2(1-\varepsilon))} }\leq C(u_0,M).
$$
We can now proceed with $w=B(w_2,w_2)$: another application of Lemma \ref{productlemma}
yields
$$
\|B(w_2,w_2)\|_{\Lb \infty {1/(1-\varepsilon)} } \lesssim\|w_2\|_{\Lb
  \infty {3/(2(1-\varepsilon))} }\|w_2\|_{\Lb \infty {3/(1-\varepsilon)}
}\,,
$$
which concludes the proof of Lemma \ref{cas2} and therefore the proof
of Proposition \ref{keyprop}.\cqfd

\medbreak
For the remaining part of this section we prove Theorem
\ref{th1}. Recall we then have $3<p<6$ and the solution $u$ satisfies a
priori bound \eqref{eq:apriori}.

In order to compensate for the lack of positive regularity on the
linear flow $u_L$, we need one further iteration: set $w_{2}=B(u_{L},u_{L})+w_{3}$, then
\begin{multline}
\label{eq:defw3}
  w_{3}=2B(u_{L},B(u_{L},u_{L}))+B(B(u_{L},u_{L}),B(u_{L},u_{L}))\\
{}+2B(u_{L},w_{3})+2B(B(u_{L},u_{L}),w_{3})+B(w_{3},w_{3}).
\end{multline}
We start with terms involving only the linear flow: standard heat
estimates yield (see e.g. Proposition 4.1 of\ccite{gip3})
$$
B(u_{L},u_{L}) \in \Lb \infty {p/2}\cap \Lb 1 {p/2}\,;
$$
then, by standard product rules, with $p<q<6$, where $\kappa=3/p-3/q>0$
is understood to be small, 
$$
B(u_{L},B(u_{L},u_{L})) \in \Lb \infty{3q/(q+3)}\cap \Lb 1{3q/(q+3)},
$$
as the worst case is when low frequencies are on $B(u_{L},u_{L})\in
\Lh \infty {-\kappa}$. Notice
that for $q=6$ we would get $\Lh \infty {1/2}=\Lb \infty 2$. The
quadrilinear term is dealt with in a similar way.
\begin{remark}
\label{rem-it}
This may be iterated again, of course, but we will not do so
here. Our restriction on $p$ comes from the balance between regularity
$3/q$ (on the trilinear term in $u_L$) and $3/p-1$ (our a priori
bound), which requires $3/q+3/p-1>0$.
\end{remark}
Next, we prove the following proposition, which is a slight improvement over
the statement from Theorem \ref{th1}.
\begin{prop}
\label{th1m}
Assume \eqref{eq:apriori} on $u$ for $3<p<6$, then, for $p<q<6$,
$$
\|w_{3}\|_{\Lb \infty {3q/(q+3)}} \lesssim C(\|u\|_{\Lb \infty{p}}).
$$
\end{prop}
We already dealt
with terms involving only $u_{L}$ in \eqref{eq:defw3}. All $B(\cdot,\cdot)$ terms
involving $w_{3}$ itself are like $B(v,w_{3})$, where $v\in \Lb\infty p$
and $\|v\|_{\Lb \infty p} \lesssim M=\|u\|_{\Lb \infty p}$. 
\begin{lemma}
Let $r$ be such that $3/r=(q+3)/q-\e$ and
$\e<6/q-1$. Let $v$ be in~$\Lb \infty p$ and~$w_3$ in~$\Lb \infty r$, then 
\begin{equation}
  \label{eq:finipdb}
  \| B(v,w_3)\|_{\Lh \infty {3/q} } \lesssim \|v\|_{\Lb \infty p}
  \|w_3\|_{\Lb \infty r}.
\end{equation}
\end{lemma}
The lemma is again a direct consequence of product rules and
properties of $B$.\cqfd

By convexity of Besov norms,
$$
\|w_{3}\|_{B_{r}} \lesssim \|w_{3}\|_{B^{3/q}}^{1-\eta} \|w_{3}\|^\eta_{B_{p}}
$$
with $\eta=\e/(1-\kappa)$, and
$$
\|B(v,w_{3})\|_{\Lh \infty {3/q}} \lesssim M^{1+\eta}
\|w_{3}\|_{L^\infty B^{3/q}}^{1-\eta}  \lesssim C(\eta)
M^{\frac{1+\eta} \eta}+\gamma(\eta) \|w_{3}\|_{\Lh \infty {3/q}}\,,
$$
where we may chose $\gamma\ll 1$. Summing estimates, we close on $w_{3}$,
$$
\|w_{3}\|_{\Lh \infty {3/ q}} \lesssim C(\delta, M)+\delta
\|w_{3}\|_{\Lh \infty{3/q}},
$$
with a small suitable $\delta$.
\begin{remark}
  Note that we assume that $u_{0}$ is actually in $\dot
  B^{3/p-1,q}_{p}$ with $q<+\infty$; then a local in time strong
  solution exists, and the a priori bound is valid as long as the
  strong solution exists, because $w_{3}$ is already known to be in
  $B^{3/p}$ as a byproduct of local existence theory. We are {\bf not}
  constructing $w_{3}$, merely improving a bound.
\end{remark}

\section {H\"older regularity in time}
\subsection{Scaled energy estimates}
Consider a local in time solution $u$ such that
$u_{0} \in \Be {-1/4} 4 4$. Assuming it exists past time $T$, one may
prove that $\sup_{0<t<T} t^{\frac 1 8} \|u(\cdot,t)\|_{L^4} <+\infty $;
from the Duhamel formula, one then obtains that $\sup_{0<t<T}
\|u-u_{L}\|_{L^2_{x}} \lesssim T^{1/4}$. The next proposition proves that
such a bound does not depend on the local Cauchy theory but only on
a suitable a priori bound:
\begin{proposition}
\label{scaledEnergy} Let~$u$ be a solution of~(NS). Then, recalling~$u= u_{L}
+w$, we have
$$
 \frac 1 {t^{\frac 1 2}} {\|w(t)\|^2_{L^2} } +\int_{0}^t  \frac 1 
{t'^{\frac 1 2}}  \Bigl( \|\nabla w(t')\|_{L^2} ^2 + \frac 1 {t'} \|w(t')\|_{L^2}^2\Bigl)
\,dt' \lesssim \|u_0\|^4_{\Be {-1/4} 4 4 } \exp \bigl (C
\|u_0\|^5_{\Be {-2/5} 5  5 }\bigr).
$$
\end{proposition}
The equation on $w$ reads
\begin{equation}
  \label{eq:mns}
  \left \{
\begin{array}{c}
\partial_{t} w -\Delta w +w\cdot \nabla w + u_{L}\cdot \nabla w = -w\cdot \nabla
u_{L}-u_{L}\cdot \nabla u_{L}-\nabla p\\ \dive w = 0\quad\hbox{and}\quad
w_{|t=0} =0.
\end{array} \right.
\end{equation}
Performing an~$L^2$ energy estimate on \eqref{eq:mns} yields
$$
\frac 1 2\frac d {dt} \|w(t)\|_{L^2}^2 +\|\nabla w(t)\|_{L^2} ^2 \leq 
\| u_{L}\|^2_{L^4}\| \nabla w(t)\|_{L^2} +\|u_L\|_{L^5} \|w \|_{L^{\frac{10}3}}\|\nabla w \|_{L^2}
$$
where integration by parts was done on all terms on the right using the divergence free condition, followed by H\"older. As $\|w\|_{10/3}\leq \|w\|_{L^2}^{2/5}\|\nabla w\|^{3/5}_{L^2}$, by convexity
\begin{equation}
\label{scaledEnergydemoeq4} \frac d {dt} \|w(t)\|_{L^2}^2 +\|\nabla
w(t)\|_{L^2} ^2 \leq 2 \|u_{L}(t)\|_{4}^4 + C \|w(t)\|_{L^2} ^2
\|u_{L}(t)\|_{5}^{5}\,.
\end{equation}
Introduce the correct scaling in time, $\psi(t)= {t^{-\frac 1 2}}
\|w(t)\|^2_{L^2}$ and let $\phi(t)=\ds \int_{0}^t
    \|u_{L}(t')\|^5_{L^5}\,dt'$,
\begin{equation*}
 \frac d {dt} \left( \psi(t) e^{-C\phi(t)}\right) + 
 \frac 1 {t^{\frac 1 2}}\left(
  \psi(t)+\|\nabla w(t)\|^2_{2} \right) e^{-C\phi(t)} 
\leq \frac 2 {t^\frac 1 2} \| u_{L}(t)\|^4_{L^4} e^{-C\phi(t)}\,.
    \end{equation*}
We now integrate over $[0,t]$,
$$
\longformule{
\psi(t)+\int_{0}^t \frac 1 {t'^{\frac 1 2}}\left(\psi(t')+ \|\nabla w(t')\|_{L^2} ^2 \right)
  \,dt' \lesssim \int_{0} ^{+\infty}
 t'^{\frac 1 2} \|S(t') u_{0}\|_{4}^4 \frac {dt'} {t'}
 }
 {  {}\times
\exp \Bigl(
C \int_{0}^{+\infty} t'\|S(t')u_{0} \|_{L^5}^5 \frac {dt'} {t'}\Bigr).  
}
$$
Recalling that in our definition of Besov norms \eqref{eq:defBesovHeat} we may replace $Q(t)$ by $S(t)$ for negative regularity, we identify equivalent norms
for the Besov norms ~$\Be {-1/4} 4 4$ and~$\Be {- 2/5} 5 5$ in our last inequality, and get the desired result. \cqfd

\subsection{From scaled energy estimates to regularity improvement}
\label{sec:sub1}
We now (re)prove a particular case of Theorem \ref{th1}, namely $p=q=4$, from the estimates in the previous subsection. Iterating the scaled energy estimate on higher order fluctuations would allow larger $p,q$.
\begin{proposition}
\label{B-1/4to1/23} Let~$u$ be a strong solution of~(NS) on a time
interval~$[0,T[$, with data~$u_{0}$ in~$ \Be {-1/4} 4 4$ and such
that $ u \in L^\infty([0,T[; \Be {-1/4} 4 4)$. Then ~$w$ is in~$L^\infty([0,T[;\Be {1/2} 2 \infty)$.
\end{proposition}
\begin{proof}
It requires an alternate definition of Besov
spaces, using discrete Littlewood-Paley decompositions rather than
heat operators.
\begin{definition}
\label{d1bis}
Let $\phi$ be a smooth function in the Schwartz class such that $\widehat\phi =
1$ for~$|\xi|\leq 1$ and $\widehat\phi= 0$ for $|\xi|>2$, and define
$
\phi_{j}(x):= 2^{dj}\phi(2^{j}x)$, and frequency localization
operators~$S_{j} := \phi_{j}\ast\cdot$, $\Delta_{j} := S_{j+1} -
S_{j}$. An equivalent definition of $\dot
B^{s,q}_{p}$ is the set of tempered distributions $f$ such that
\begin{itemize}
\item the partial sum $ \sum^{m}_{-m} \Delta_{j} f $ converges to $f$
  as a tempered distribution if~$s < {\frac{d}{p}}$ and after taking
  the quotient with polynomials if not, and
\item the sequence $\epsilon_{j} := 2^{js}\| \Delta_{j}  f \|_{p}$
  is in~$\ell^{q}$; its $\ell^q$-norm defines the Besov norm of $f$.
\end{itemize}
\end{definition}
We proceed with proving Proposition \ref{B-1/4to1/23}. From standard heat kernel bounds for frequency localized functions, \eqref{eq:nsi} yields the inequality
\begin{multline} 
\label{B-1/4to1/23demoeq0} 
2^{\frac j 2} \|\D_{j} w(t)\|_{L^2} \lesssim  \int_{0}^t e ^{-c2^{2j}(t-t')} 2^{\frac 3 2 j } \|\D_{j}(u_L\otimes u_L(t')\|_{L^2} \, dt'\\
{}+ \int_{0}^t e ^{-c2^{2j}(t-t')} 2^{\frac 3 2 j } \|\D_{j}(u_L\otimes w)(t')\|_{L^2} \, dt'+ \int_{0}^t e ^{-c2^{2j}(t-t')} 2^{\frac 3 2 j } \|\D_{j}(w\otimes w)(t')\|_{L^2} \, dt'\,.
\end{multline}
Let us denote by $K_j(t)+J_j(t)+I_j(t)$ the righthand side. The
first term is easy, using standard heat decay:
$t^{1/8}\|u_L(t)\|_4\lesssim \|u_0\|_{\dot B^{-1/4,\infty}_4}$, and

$$
K_j(t)\lesssim  \int_{0}^t e ^{-c2^{2j}(t-t')} {2^{\frac 3 2 j }} t'^{-1/4}  \, dt'=\int_0^{2^{2j}t} e^{-c (2^{2j}t -\tau)} \tau^{-1/4}\,d\tau\lesssim 1.
$$
The second term is similar, using $t^{1/2}\|u_L(t)\|_\infty\lesssim \|u_0\|_{\dot B^{-1,\infty}_\infty}$, and
$\|w(t)\|_{L^2}\lesssim t^\frac 1 4$,
$$
J_j(t)\lesssim  \int_{0}^t e ^{-c2^{2j}(t-t')} {2^{\frac 3 2 j }}
t'^{1/4} t'^{-1/2}  \, dt'=\int_0^{2^{2j}t} e^{-c (2^{2j}t -\tau)} \tau^{-1/4}\,d\tau\lesssim 1.
$$
Let us decompose~$I_{j}(t)$ by introducing~$t_{j,\Lam}
\eqdefa t-\Lam 2^{-2j}$ (where~$\Lam$ will be chosen later on) and
set $I_j(t)=I_{j,1}(t)+ I_{j,2}(t)$ with
\begin{eqnarray*}
I_{j,1}(t) & = & \int_0^{t_{j,\Lam}} e ^{-c2^{2j}(t-t')} 2^{\frac 3 2 j } \|\D_{j}(w\otimes w)(t')\|_{L^2} \, dt'\quad\hbox{and}\\
I_{j,2}(t) & = & \int_{t_{j,\Lam}}^t e ^{-c2^{2j}(t-t')} 2^{\frac 3 2 j } \|\D_{j}(w\otimes w)(t')\|_{L^2} \, dt'\,.
\end{eqnarray*}
We have
$$
I_{j,1}(t)\leq e ^{-\frac c 2 \Lam} \int_{0}^{t_{j,\Lam}} e^{-\frac c
2 2^{2j} (t_{j,\Lam}-t')} 2^{\frac 3 2 j } \|\D_{j}(u\otimes u(t')\|_{L^2} \,dt'.
$$
From product rules in Besov spaces,
$$
\|\D_{j}(w\otimes w)\|_{L^\infty (0,T;L^2)} \lesssim 2^{\frac j 2 }
\|w\|_{L^\infty (0,T;\Be {-1/4} 4 \infty)}
\|w\|_{L^\infty (0,T;\Be {1/2} 2 \infty)}.
$$
Choosing~$\Lam$ such that $C e^{-\frac c2 \Lam} \|w\|_{L^\infty
  (0,T;\Be {-1/4} 4 \infty)} \leq \frac 1 2\, \virgp$ 
we get
\begin{equation}
\label{B-1/4to1/23demoeq1} \sup_{t\in [0,T[}I_{j,1} (t) \leq \frac 1
2 \|w\|_{L^\infty (0,T;\Be {1/2} 2 \infty)}.
\end{equation}

We are left with~$I_{j,2}(t)$. We may replace $w$ by $u$, as this just
adds terms which are similar to the $K_j$ and $J_j$ terms. We then split~$u$ on the interval~$[t_{j,\Lam} ,t]$ in the following way
$$
u = u_{L,j} +w_{j} \quad\hbox{with}\quad u_{L,j}(t) \eqdefa
e^{(t-t_{j,\Lam})\D} u(t_{j,\Lam}).
$$
By the same reasonning that took care of the $K_j$ and $J_j$ terms,
the $u_{L,j}\otimes u_{L,j}$ and $u_{L,j}\otimes w_j$ terms in $I_j$
are uniformly bounded.
We are left with quadratic terms~$w_{j}\otimes w_{j}$. Using Bernstein inequality we have
$$
\|\D_{j}(w_{j}\otimes w_{j})\|_{L^2} \lesssim 2^{\frac 3 2 j}
\|w_{j}(t)\|^2_{L^2}.
$$
By integration on the interval~$[t_{j,\Lam}, t]$ (the length of which
is less than~$\Lam 2^{-2j}$)
$$
2^{\frac {3j} 2 } \int_{t_{j},\Lam} ^t \|\D_{j} (w_{j} (t')\otimes
w_{j}(t'))\|_{L^2} dt' \lesssim 2^{3j} (\Lam 2^{-2j})\|w_{j}\|_{L^\infty
[t_{j,\Lam},t];L^2)}^2.
$$
Proposition\refer{scaledEnergy} with initial time~$t_{j,\Lam}$ implies
$$
2^{\frac {3j} 2 } \int_{t_{j},\Lam} ^t \|\D_{j} (w_{j} (t')\otimes
w_{j}(t'))\|_{L^2} dt' \lesssim \Lam \|u\|^4_{L^\infty(0,T; \Be
  {-1/4}4 4)} \exp \bigl(C\|u\|^5 _{L^\infty(0,T; \Be{-2/5} 5 5)}\bigr).
$$
Then plugging all this in\refeq{B-1/4to1/23demoeq0}
and\refeq{B-1/4to1/23demoeq1}, we get, for any~$t<T$,
$$
\longformule{ \|w\|_{L^\infty([0,t]; \Be{1/2}2 \infty)} \lesssim
\|u_0\|^2_{\Be{-1/4} 4 \infty} + \frac 1 2 \|w\|_{L^\infty([0,t];
\Be{1/2} 2 \infty)} } { {} + C(1+\Lam) \|u\|^4
_{L^\infty(0,T; \Be{-1/4} 4 4)} \exp \bigl(C\|u\|^5
_{L^\infty(0,T; \Be{-2/5} 5 5)}\bigr).  }
$$
The choice of~$\Lam$ means that $\Lam \sim \log \bigl(
e+\|u\|_{L^\infty(0,T; \Be{-1/4} 4 4)}\bigr)$.  This
concludes the proof of Proposition\refer{B-1/4to1/23}. %\hfill $\Box$
\end{proof}

Let us prove Theorem\refer{thholder}. Let us consider two times~$t$ and~$t_0$  in~$[0,T[$. We can assume that~$t_0<t$. Then, let us write that
$$
u(t)-u(t_0) = u(t)-S(t-t_0)u(t_0) + \bigl(S(t-t_0)-\Id\bigr) u(t_0).
$$
Applying Proposition\refer{B-1/4to1/23} at time~$t_0$ and using
$L^2\hookrightarrow \Be{-3/4} 4 4$
$$
\|u(t) -S(t-t_0)u(t_0)\|_{\Be{-3/4} 4 4} \leq C(M) |t-t_0|^{\frac 1 4}.
$$
Moreover, we have
$$
\bigl\|\bigl(S(t-t_0)-\Id\bigr)u(t_0)\|_{\Be {-3/4} 4 4} \leq
C|t-t_0|^{\frac 1 4} \|u(t_0)\|_{\Be{-1/4} 4 4}\,,
$$
and Theorem\refer{thholder} is proved. \cqfd


\begin{thebibliography}{50}

\bibitem{bcdbook}
H. Bahouri, J.-Y. Chemin and R. Danchin, 
{\it Fourier Analysis and Nonlinear Partial Differential Equations}, Grundlehren der mathematishcen Wissenschaften, {\bf 343}, Springer-Verlag Berlin Heidelberg, 2011.


\bibitem{Cannone}
M.Cannone, A generalization of a theorem by Kato on Navier-Stokes equations,  {\it Revista Matematica Iberoamericana}, {\bf 13}, 1997, pages 515--541. 

\bibitem{CP}
M. Cannone and F.Planchon, On the regularity of the bilinear term for solutions to the incompressible Navier-Stokes equations.  {\it Revista Matematica Iberoamericana}, {\bf 16}, 2000, pages 1--16. 

\bibitem{chemin20}
J.-Y. Chemin,
Th\'eor\`emes d'unicit\'e pour le syst\`eme de Navier-Stokes tridimensionnel.
{\em Journal d'Analyse Math\'ematique}, {\bf 77}, 1999, pages 27--50.

\bibitem{Ches}
A. Cheskidov and R.  Shvydkoy, 
The regularity of weak solutions of the 3{D} {N}avier-{S}tokes  equations in $B^{-1}_{\infty,\infty}$,
{\it Archiv for Rational  Mechanics and  Analysis}, {\bf 195}, 2010, pages 159--169.

\bibitem{ess} L. Escauriaza, G. A. Seregin, and V. \v Sver\'ak, $ L_{3,\infty}$ solutions of Navier-Stokes equations and backward 
uniqueness, {\it  Uspekhi Mat. Nauk}, {\bf  58}, 2003, pages 3--44. 

\bibitem{gip3}
I. Gallagher, D. Iftimie, and F. Planchon. Asymptotics and stability for global solutions to the Navier- 
Stokes equations,  {\it Annales de l'Institut Fourier}, {\bf  53} , 2003, pages 1387 --1424.

\bibitem{GKP}
I. Gallagher, G. Koch and F. Planchon,  A profile decomposition approach to the $L^\infty_t (L^3_x )$ 
Navier-Stokes regularity criterion, {\em Math. Ann.} (published online July 2012), DOI: 10.1007/s00208-012-0830-0.

\bibitem{kato} T. Kato: Strong~$L^p$ solutions of the Navier--Stokes equations
 in~$\R^m$ with applications to weak solutions, {\em Mathematische Zeitschrift},  {\bf 187},
 1984, pages 471--480.

\bibitem {KF} T. Kato and H. Fujita,  On the nonstationary Navier-Stokes system, {\it  Rend. Sem. Mat. Univ. 
Padova}, {\bf  32}, 1962, pages 243--260.

\bibitem{KT}  H.  Koch and D. Tataru, Well-posedness for the Navier-Stokes equations,  Advances in  Mathematics, {\bf 157}, 2001, pages 22-- 35. 

\bibitem{leray}
J. Leray, Essai sur le mouvement d'un liquide visqueux emplissant l'espace,
{\em Acta Matematica}, {\bf 63}, 1933, pages 193--248.

\bibitem{Planchon} F. Planchon,  Asymptotic behavior of global solutions to the Navier-Stokes equations in $\R^3$, {\it Revista Matematica Iberoamericana}, {\bf 14}, 1998, pages 71--93. 

\end{thebibliography}
\end{document}